\documentclass[oneside, reqno, 10pt]{amsart}
\usepackage{stmaryrd}

\makeatletter
\renewcommand{\l@section}{\@tocline{1}{0pt}{1pc}{1pc}{}} 
\makeatother

\usepackage{titlesec}

\titleformat{\section}
 {\normalfont\Large\bfseries}
 {\thesection}{1em}{}

\titleformat{\subsection}
 {\normalfont\large\bfseries} 
 {\thesubsection}{1em}{}

\titleformat{\subsubsection}
 {\normalfont\normalsize\bfseries}
 {\thesubsubsection}{1em}{}

\author[]{Akihiro Munemasa}
\address{Graduate School of Information Sciences,
Tohoku University, Aoba Ward, Sendai, 980-8579, Japan}
\email{munemasa@tohoku.ac.jp}
\thanks{The research of A.~Munemasa was supported by JSPS KAKENHI Grant Number 25K07095.}

\author[]{Ferenc Sz\"{o}ll\H{o}si}
\address{Department of Mathematical Sciences, Interdisciplinary Faculty of Science and Engineering, Shimane University, 1060 Nishikawatsucho, Matsue, Shimane, Japan}
\email{szollosi@riko.shimane-u.ac.jp}
\thanks{The research of F.~Sz\"{o}ll\H{o}si was supported by JSPS KAKENHI Grant Number 24K06829.}

\author[]{Kiyoto Yoshino}
\address{Department of Information Science, Faculty of Science, Toho University, 2-2-1 Miyama, Funabashi, Chiba 274-8510, Japan}
\email{kiyoto.yoshino@is.sci.toho-u.ac.jp}

\usepackage{amsmath,amsthm,amssymb,amsfonts,color,comment,stix}
\usepackage[T1]{fontenc}
\usepackage{times}
\usepackage[showonlyrefs]{mathtools}
\mathtoolsset{showonlyrefs}

\usepackage[shortlabels]{enumitem}
\setlist[enumerate,1]{label={\upshape(\roman*)}}

\newcommand{\nexteq}{\displaybreak[0]\\ &=}

\numberwithin{equation}{section}

\newtheorem{lemma}{Lemma}[section]
\newtheorem{proposition}[lemma]{Proposition}%[section]
\newtheorem{theorem}[lemma]{Theorem}%[section]
\newtheorem{corollary}[lemma]{Corollary}%[section]

\theoremstyle{definition}
\newtheorem{definition}[lemma]{Definition}%[section]
\newtheorem{remark}[lemma]{Remark}%[section]
\DeclareMathOperator{\Aut}{Aut}
\DeclareMathOperator{\disc}{disc}

\DeclareMathOperator{\sw}{sw}

\DeclareMathOperator{\Sym}{Sym}
\DeclareMathOperator{\TD}{TD}
\newcommand{\cB}{\mathcal{B}}

\newcommand{\aes}{\omega}

\newcommand{\cX}{\mathcal{X}}

\newcommand{\cS}{\mathcal{S}}
\newcommand{\cC}{\mathcal{C}}

\newcommand{\cL}{\mathcal{L}}

\newcommand{\R}{\mathbb{R}}
\newcommand{\Z}{\mathbb{Z}}

\newcommand{\cM}{\mathcal{M}}
\newcommand{\latin}{S}

\newcommand{\sr}{\rho} %{(0,0,2\alpha_1)}
\usepackage{bbm} % preamble
\newcommand{\one}{\mathbbm{1}}

\title[]{Sets of equiangular lines in dimension $18$ constructed from $A_5^3 \oplus A_1^4$}

\begin{document}
\begin{abstract}
	In 2023, Greaves, Syatriadi, and Yatsyna found a set of $57$ equiangular lines in $\mathbb{R}^{18}$, breaking the previous record.
	In 2025, Lin, Munemasa, Taniguchi, and Yoshino constructed a large number of sets of $57$ equiangular lines in $\mathbb{R}^{18}$ as affine equiangular sets in an integral overlattice of $A_9^2 \oplus A_1$.
	In this paper, we construct further sets of $57$ equiangular lines in $\mathbb{R}^{18}$ from Latin squares of order $6$ and Pasch configurations, realized as affine equiangular sets in an integral overlattice of $A_5^3 \oplus A_1^4$.
	Unlike the previously known examples, these sets are not strongly maximal.
	Moreover, some of them have only five distinct Seidel eigenvalues, fewer than any previously known examples.
\end{abstract}

\maketitle
%\tableofcontents

\section{Introduction}
Equiangular lines in Euclidean spaces form a classical topic in combinatorics.
A set of lines through the origin in a Euclidean space is called \emph{equiangular} if any two distinct lines form the same angle.
Since the early work of Haantjes in the 1940s~\cite{Haantjes1948}, a central program has been to determine the maximum number $N(d)$ of equiangular lines in $\mathbb{R}^d$.
The values of $N(d)$~\cite{Greaves2022, Greaves2020, Greaves2023, Haantjes1948, lemmens1973, Lint1966} and bounds for $N(d)$~\cite{de2000, GHAF2016, lemmens1973} have been studied.
For convenience, we summarize the known values or best bounds for $d \leq 23$ in Table~\ref{table:N(d)},
where $\arccos(1/\alpha)$ denotes the common angle of a set of equiangular lines in $\mathbb{R}^d$ attaining $N(d)$.

\begin{table}[hbtp]
	\caption{The values or bounds of $N(d)$ for $d \leq 23$.}
	\label{table:N(d)}
	\centering
	\begin{tabular}{c|ccccccccccccccc}
		$d$    & 2 & 3          & 4              & 5  & 6  & 7--14   & 15 & 16 & 17 & 18     & 19     & 20     & 21  & 22  & 23  \\
		\hline
		$N(d)$ & 3 & 6          & 6              & 10 & 16 & 28      & 36 & 40 & 48 & 57--59 & 72--74 & 90--94 & 126 & 176 & 276 \\
		$\alpha$
		       & 2 & $\sqrt{5}$ & $\sqrt{5},\,3$ & 3  & 3  & $3,\,5$
		       & 5 & 5          & 5              & 5  & 5  & 5       & 5  & 5  & 5
	\end{tabular}
\end{table}

For a fixed angle $\arccos(1/\alpha)$, a more precise problem is to determine the maximum number $N_{\alpha}(d)$ of equiangular lines in $\mathbb{R}^d$ with common angle $\arccos(1/\alpha)$.
It is known that if $N_{\alpha}(d) > 2d$, then $\alpha$ must be an odd integer~\cite{lemmens1973}.
Consequently, much of the literature focuses on the cases where $\alpha$ is an odd integer.

In high dimensions, Jiang et al.\ proved that for every odd integer $\alpha \geq 3$,
$N_{\alpha}(d) = \lfloor (\alpha+1)(d-1)/(\alpha-1) \rfloor$
for all sufficiently large $d$~\cite{Zhao2021}.
More precise results are known for the angles $\arccos(1/3)$ and $\arccos(1/5)$.
For the common angle $\arccos(1/3)$, Lemmens and Seidel~\cite{lemmens1973} completely determined the values of $N_{3}(d)$.
Specifically, $N_{3}(d) = 2(d-1)$ for all $d \geq 15$.
For the common angle $\arccos(1/5)$, the so-called Lemmens--Seidel conjecture,
$N_{5}(d) = \max\{276, \lfloor (3d-3)/2 \rfloor\}$ for all $d \geq 23$,
has been proved in~\cite{cao2022,lemmens1973,yoshino2022}.

While the values $N_{\alpha}(d)$ in high dimensions are well understood~\cite{Zhao2021},
the exact values of $N_{\alpha}(d)$ remain unknown in low dimensions for each fixed angle.
Among the cases where $\alpha$ is odd, the only nontrivial angle for which the problem has been completely resolved is $\arccos(1/3)$~\cite{lemmens1973}.
In this case, an infinite family attains $N_{3}(d) = 2(d-1)$ for all $d \geq 15$,
and the remaining maximum sets of equiangular lines for $d \leq 14$ have been explicitly determined~\cite{cao2022}.
The former can be constructed in a simple manner~\cite{lemmens1973}, whereas the latter are realized in dimension $7$ as subsets of the equiangular line set $\Psi_{7}$ associated with the $E_8$ root system, which attains $N_{3}(7) = 28$~\cite{cao2022}.

Turning to the angle $\arccos(1/5)$, the set $\Psi_{23}$ of $N_{5}(23)$ equiangular lines in dimension $23$ is known to be unique~\cite{goethals1975}.
This configuration may be viewed as an analogue of $\Psi_{7}$ for the angle $\arccos(1/5)$ in the sense that it is in fact related to the Leech lattice.
However, in contrast to the case of $\arccos(1/3)$, there exist sets of equiangular lines with common angle $\arccos(1/5)$ in dimension $18$ that are not contained in $\Psi_{23}$~\cite{yoshino2022}.
As a result, the situation is considerably more intricate, and as shown in Table~\ref{table:N(d)}, the exact values of $N_{5}(d)$ are not known for $d = 18, 19, 20$.

Motivated by this situation,
in this paper we investigate sets of equiangular lines with common angle $\arccos(1/5)$ in dimension $18$.
A series of improvements of upper bounds~\cite{Greaves2022, lemmens1973} and lower bounds~\cite{Greaves2023,Lin2020b,Szollosi2019} on $N_{5}(18)$ has been obtained by a variety of methods.
Earlier constructions were based on removing suitable lines from $\Psi_{23}$,
whereas more recent approaches rely on searches for integer vectors with prescribed inner products.
These methods have led to the bound $N_{5}(18) \geq 57$.
Moreover, such configurations are known to be strongly maximal, in the sense that they cannot be extended even in higher dimensions~\cite{yoshino2022}.
Furthermore, numerous non-isometric examples of $57$ equiangular lines in dimension $18$ have been constructed~\cite{LMTY2025} as affine equiangular sets in an integral overlattice of $A_9^2\oplus A_1$.
In the meantime, the second author found a set of $57$ equiangular lines in dimension $18$ by extending a (maximal) set of $48$ equiangular lines in dimension $17$.
It turns out that this set of $57$ equiangular lines gives rise to an affine equiangular set in an integral overlattice of $A_5^3\oplus A_1^4$.
Our main result, Theorem~\ref{thm:main}, determines all maximum affine equiangular sets lying in this overlattice with respect to the specified switching root.
This shows that the corresponding sets of $57$ equiangular lines in dimension $18$ can be constructed from a Latin square of order $6$ giving rise to $36$ equiangular lines, together with three sets of $7$ equiangular lines each coming from a Pasch configuration.
Moreover, we prove in Theorem~\ref{thm:maximality} that these examples are maximal,
in the sense that they cannot be extended within the same dimension.
However, it turns out that, once the dimension is increased, further lines can be added to each of these examples.
In particular, we show in Theorem~\ref{thm:stronglymaximality} that they are not strongly maximal.
These appear to be the first known examples of $57$ equiangular lines in dimension $18$ with this property.
Also, in Corollary~\ref{cor:characteristic}, we obtain sets of equiangular lines whose associated Seidel matrices have
only five distinct eigenvalues,
which is one fewer than in previously known constructions in~\cite{LMTY2025}.
It is worth mentioning that despite our extensive random generation of $57$ equiangular lines in dimension $18$, we were not able to find any equiangular lines not covered by the classification done in~\cite{LMTY2025} or in the present paper.

Our approach uses the framework of affine equiangular sets to relate sets of equiangular lines to combinatorial objects such as Latin squares.
For details on affine equiangular sets, we refer to~\cite{LMTY2025}.
Briefly, a set of equiangular lines with common angle $\arccos(1/5)$ is spanned by vectors of norm $5$ with pairwise inner products $\pm 1$.
After a suitable affine transformation, these vectors become a set of vectors of norm $3$ with pairwise inner products $0$ or $1$, which we call an affine equiangular set.
Since the lattice generated by an affine equiangular set often has a smaller discriminant, it can be easier to handle.
In fact, it is often the case that the resulting lattice contains a root sublattice of full rank. This allows us to describe affine equiangular sets using root vectors.

This paper is organized as follows.
In Section~\ref{sec:preliminaries}, we introduce some combinatorial objects and basic definitions on lattices.
In Section~\ref{sec:A}, we describe an integral overlattice of $A_5^3\oplus A_1^4$, in which we found affine equiangular sets of cardinality $57$.
In Sections~\ref{sec:beta0} and \ref{sec:beta123}, we explain our constructions in terms of combinatorial objects.
In Section~\ref{sec:main}, we summarize the previous sections and determine all maximum affine equiangular sets in the overlattice with respect to the specified switching root.
In Section~\ref{sec:maximality}, we prove that these affine equiangular sets are maximal.
Finally, in Section~\ref{sec:strongmaximality}, we show that these are not strongly maximal.

\section{Preliminaries} \label{sec:preliminaries}
In this section, we introduce some combinatorial objects and basic definitions on lattices.

\subsection{Combinatorics}

In this section, we introduce some combinatorial objects used in this paper.
Using them, we will describe affine equiangular sets.

\begin{definition}[Transversal design]
	A \emph{transversal design} $\TD(k,n)$ is a triple $(X,\mathcal{G},\mathcal{B})$ satisfying the following conditions:
	\begin{enumerate}
		\item $X$ is a set of $kn$ points.
		\item $\mathcal{G}$ is a partition of $X$ into $k$ $n$-subsets.
		\item $\mathcal{B}$ is a set of $k$-subsets of $X$ such that $|B \cap G| = 1$ for any $B \in \mathcal{B}$ and $G \in \mathcal{G}$.
		\item For any $x \in G$ and $y \in G'$ with $G \neq G' \in \mathcal{G}$, there exists a unique $B \in \mathcal{B}$ such that $\{x,y\} \subset B$.
	\end{enumerate}
	The elements of $\mathcal{G}$ are called \emph{groups} and the elements of $\mathcal{B}$ are called \emph{blocks}.
\end{definition}

\begin{definition}	\label{dfn:latin}
	A \emph{Latin square} is an $n \times n$ array filled with $n$ different symbols, each occurring exactly once in each row and exactly once in each column.
\end{definition}
The group $S_3 \ltimes S_n^3$ naturally acts on the set of Latin squares of order $n$ by permuting the three roles ``rows'', ``columns'' and ``symbols'' and permuting the rows, columns and symbols.
Also, a Latin square $L$ of order $n$ induces the transversal design $\TD(3,n)$ with groups
$\{\{1,\ldots,n\},\{n+1,\ldots,2n\},\{2n+1,\ldots,3n\}\}$
and blocks $\{\{i,n+j,2n+L_{i,j}\} : i,j \in [n]\}$.
Conversely, the Latin square can be recovered from the transversal design.
Thus, we may identify Latin squares of order $n$ with transversal designs $\TD(3,n)$.

In particular, the transversal design $\TD(3,2)$ is known as a \emph{Pasch configuration}.
Here, a \emph{Pasch configuration} is an incidence structure $(V,\cB)$ satisfying:
\begin{enumerate}
	\item $\lvert V\rvert = 6$ and $\lvert \cB\rvert = 4$,
	\item each point $v\in V$ lies on exactly two blocks,
	\item each block $B\in\cB$ contains exactly three points,
	\item (Linearity) any two distinct blocks intersect in at most one point.
\end{enumerate}
Indeed, the sets $\{1,3,5\}$, $\{1,4,6\}$, $\{2,3,6\}$, and $\{2,4,5\}$ form a Pasch configuration.
This is also regarded as a transversal design $\TD(3,2)$ with $3$ groups $\{1,2\}$, $\{3,4\}$, and $\{5,6\}$.
The corresponding Latin square of order $2$ with two symbols is
$\left[
		\begin{smallmatrix}
			1 & 2 \\
			2 & 1
		\end{smallmatrix}
		\right]$.
Also, the complements in $[6]$ of the four blocks of this Pasch configuration form another Pasch configuration,
and it induces the Latin square
$\left[
		\begin{smallmatrix}
			2 & 1 \\
			1 & 2
		\end{smallmatrix}
		\right]$.
These are the only two Latin squares of order $2$,
and hence, with a fixed set of groups, there are only two transversal designs $\TD(3,2)$.

Also, we introduce some basic definitions on graphs.
We consider only undirected graphs, without loops or multiple edges.
For two graphs $G$ and $H$, the cartesian product $G \boxempty H$ of $G$ and $H$ is the graph with vertex set $V(G) \times V(H)$, where two vertices $(u,v)$ and $(u',v')$ are adjacent if and only if either $u=u'$ and $v$ is adjacent to $v'$, or $v=v'$ and $u$ is adjacent to $u'$.
The \emph{distance} between two vertices $u$ and $v$ in a connected graph is defined to be the length of a shortest path connecting $u$ and $v$.
For a connected graph $G$, the distance $2$-or-$3$ graph of $G$ is the graph with the same vertex set as $G$, where two vertices are adjacent if and only if their distance in $G$ is either $2$ or $3$.
For positive integers $d$ and $n$, the Hamming graph $H(d,n)$ is the cartesian product of $d$ copies of the complete graph $K_n$ on $n$ vertices.

\subsection{Lattices}
In this subsection, we introduce some basic definitions and notation on lattices.
For more information, see~\cite{E}.
A \emph{lattice} is the set $L$ of integral linear combinations of some $\R$-linearly independent vectors $u_1,\ldots,u_n \in \R^N$ ($n \leq N$).
This lattice is written as $\langle \{u_1,\ldots,u_n\} \rangle$.
For two lattices $L$ and $M$ with $L \supset M$, we say that $M$ is a sublattice of $L$, and that $L$ is an overlattice of $M$.
Also, denote by $(u,v)$ the inner product of two vectors $u$ and $v$.
For two lattices $L$ and $M$, we write $L \oplus M$ for the orthogonal sum of $L$ and $M$.
The orthogonal sum of $n$ copies of a lattice $L$ is denoted by $L^n$.

A lattice $L$ is said to be \emph{integral} if $(u,v) \in \Z$ for any $u,v \in L$.
For an integral lattice $L$, the \emph{dual} $L^*$ of $L$ is defined to be the set $\{v \in \R L \mid (u,v) \in \Z \text{ for any } u \in L\}$,
where $\R L$ is the $\R$-linear span of $L$.
The \emph{discriminant} of an integral lattice $L$ is defined to be the order of the quotient group $L^*/L$, and is denoted by $\disc L$.
For each vector $u$, we call $(u,u)$ the \emph{norm} of $u$.
A vector is called a \emph{root} if its norm is $2$.
An integral lattice is called a \emph{root lattice} if it is generated by roots.
It is well known that irreducible root lattices are classified into three types $A$, $D$ and $E$.

Next, we describe the root lattices $A_n$ of type $A$ and rank $n$ and their dual.
For detailed explanations, see \cite{Conway1999, Ozeki2018}.
The root lattice $A_n$ is defined as $\{(x_1,\ldots,x_{n+1}) \in \Z^{n+1} \mid x_1 + \cdots + x_{n+1} = 0\}$.
Let
\[
	\alpha_n=\left(\frac{1}{n+1},\dots,\frac{1}{n+1},-\frac{n}{n+1}\right)\in A_n^*.
\]
For $i\in[6]$, let
\[
	\alpha_5^{(i)}=\begin{cases}
	\alpha_5&\text{if $i=6$,}\\
	\alpha_5^{(i\;6)}&\text{otherwise,}
	\end{cases}
\]
where $(i\;6) \in S_6$ is a transposition on the set $[6]$
acting on $\R^6$ by coordinate permutation.

\begin{lemma}[\cite{Ozeki2018}] \label{lem:minA}
	For each $m \in [n+1]$, the $\binom{n+1}{m}$ minimum representatives of $m \alpha_n + A_n \in A_n^*/A_n$ are given by
	\begin{align}\label{enum_min}
		\frac{m}{n+1} \sum_{i \in [n+1] \setminus I} e_i-\frac{n+1-m}{n+1} \sum_{i \in I} e_i \qquad ( I \in \binom{[n+1]}{m} )
	\end{align}
	with norm
	\begin{align}\label{min}
		\frac{m(n+1-m)}{n+1}.
	\end{align}
\end{lemma}
For a set of vectors $V$, we write $\min V$ for the minimum norm of vectors in $V$.
By Lemma~\ref{lem:minA}, we have
\begin{align}
	\min\left(\alpha_5+A_5 \right)
	                                        & =\min\left(5\alpha_5+A_5\right)=\frac56,\label{A5min}  \\
	\min\left(2\alpha_5+A_5\right)
	                                        & =\min\left(4\alpha_5+A_5\right)=\frac43,\label{A5min2} \\
	\min\left(3\alpha_5+A_5\right)
	                                        & =\frac32,\label{A5min3}                                \\
	\min\left(\alpha_1+A_1\right)
	                                        & =\frac12,\label{A1min}                                 \\
	\{x\in\alpha_5+A_5\mid (x,x)=\frac56\}  & =\{\alpha_5^{(i)}\mid i\in[6]\},
	\label{A5mv}                                                                                     \\
	\{x\in3\alpha_5+A_5\mid (x,x)=\frac32\} & =\{\alpha_I\mid I\in\binom{[6]}{3}\},
	\label{A5mv3}                                                                                    \\
	\{x\in\alpha_1+A_1\mid (x,x)=\frac12\}  & =\{\pm\alpha_1\},
	\label{A1mv}
\end{align}
where
\begin{align}\label{A5mv_def}
	\alpha_I=\frac12
	(\sum_{i\in[6]\setminus I}e_i-\sum_{i\in I}e_i)\quad
	(I\in\binom{[6]}{3}).
\end{align}

\subsection{Affine equiangular sets}
In this subsection, following \cite{LMTY2025},
we introduce the notion of affine equiangular sets.
For proofs, see~\cite{LMTY2025}.

\begin{definition}[{\cite[Definition~2.1]{LMTY2025}}]	\label{dfn:aes}
	Let $s\geq1$ be a real number, and let $r$ be a root in $\R^N$, where $N \in \Z_{\geq 1}$.
	The set of vectors $u_1,\ldots,u_n$ in $\R^N$ is called an \emph{affine equiangular set with norm $s$ with respect to $r$} if
	\begin{enumerate}
		\item $(u_1,r) = \cdots = (u_n,r) = 1$,
		\item $(u_i, u_j) = \begin{cases}
				      s   & \text{ if } i = j, \\
				      0,1 & \text{ otherwise.}
			      \end{cases}$
	\end{enumerate}
	%For simplicity, we will often refer to this as an affine equiangular set.
	The root $r$ is called the \emph{switching root} of the affine equiangular set.
\end{definition}

Affine equiangular sets are closely related to equiangular lines.
Let $\aes=\{u_1,\ldots,u_n\}$ be an affine equiangular set with norm $s$ with respect to a switching root $r$.
For each $u \in \aes$, set
\begin{align}\label{eq:tilde_def}
	\widetilde{u}:=\frac{2u-r}{\sqrt{2}}.
\end{align}
Then a direct computation gives
\begin{align}\label{eq:tilde}
	(\widetilde{u_i},\widetilde{u_j})=
	\begin{cases}
		2s-1  & \text{if } i=j,   \\
		\pm 1 & \text{otherwise.}
	\end{cases}
\end{align}
Hence, this implies that the matrix obtained from the Gram matrix of
$\widetilde{u_1},\ldots,\widetilde{u_n}$ by modifying its diagonal entries to zero
is a Seidel matrix.
Furthermore,
\[
	\R \widetilde{u_1} ,\ldots,\R \widetilde{u_n}
\]
form a set of equiangular lines in the subspace
$\langle u_1,\ldots,u_n,r\rangle_{\R}\cap r^\perp$ with common angle
$\arccos(1/(2s-1))$.
We say that this set of equiangular lines is induced by $\aes$.
Conversely, any set of equiangular lines with common angle $\arccos (1/(2s-1))$
can be induced by an affine equiangular set of norm $s$.

\begin{definition}\label{dfn:switching}
	Let $\aes$ and $\aes'$ be affine equiangular sets with respect to a common switching root $r$.
	We say that $\aes$ and $\aes'$ are \emph{switching equivalent}, and write $\aes\sim_{\sw}\aes'$,
	if there exists a bijection $g:\aes\to\aes'$ such that $g(u)\in\{u,r-u\}$ for every $u\in\aes$.
	The \emph{switching class} of $\aes$ is defined to be the set $[\omega]$ consisting of all affine equiangular sets switching equivalent to $\aes$.
\end{definition}

Switching equivalent affine equiangular sets induce the same set of equiangular lines.
In this paper, we say that a mapping between sets of vectors is an \emph{isometry} if it preserves inner products.

\begin{definition}\label{dfn:isometry}
	Let $\aes$ be an affine equiangular set with a switching root $r$ and let $\aes'$ be one with a switching root $r'$.
	We say that $\aes$ and $\aes'$ are \emph{isometric} if there exists an affine equiangular set $\aes''$
	with $\aes''\sim_{\sw}\aes'$ and a bijective isometry
	\[
		h:\aes\cup\{r\} \to \aes''\cup\{r'\}
	\]
	such that $h(r)=r'$.
\end{definition}

It is known that two affine equiangular sets are isometric if and only if their induced sets of equiangular lines are isometric; see~\cite[Lemma~2.4]{LMTY2025}.
Also, let $\aes$ be an affine equiangular set with switching root $r$ such that $\omega \cup \{r\}$ is contained in a set $V$.
Then, we say that $\aes$ is \emph{maximal} in $V$ if there is no affine equiangular set in $V$ with respect to $r$ properly containing $\aes$.
Here, the set $V$ will be taken to be a finite set of vectors, a lattice, or a Euclidean space.
If $V$ is not specified, it is the $\R$-linear span of $\aes\cup\{r\}$.
In the case where $V$ is this $\R$-linear span, the induced set of equiangular lines is maximal in the $\R$-linear span of these lines.
Similarly, we say that $\aes$ is \emph{maximum} in $V$ if there is no affine equiangular set $\omega'$ in $V$ with respect to $r$ such that $|\omega'| > |\omega|$.

\begin{definition}
	An affine equiangular set $\aes$ with respect to a switching root $r$ is said to be \emph{strongly maximal},
	if there is no affine equiangular set with respect to $r$ properly containing $\aes$ even if the dimension of the underlying space is increased.
\end{definition}

We see that an affine equiangular set $\aes$ with norm $s$ with respect to a switching root $r$ is strongly maximal
if and only if the induced set of equiangular lines is strongly maximal, that is, there is no set of equiangular lines properly containing it.
%Also, we have the following criterion for strong maximality.

%\begin{lemma}[{\cite[Lemma~2.8]{LMTY2025}}]\label{lem:strongly_maximality_criterion}
%	Let $\aes$ be an affine equiangular set with norm $s$ with respect to $r$.
%	Then $\aes$ is strongly maximal if and only if there is no vector
%	$v\in \langle \aes\cup\{r\}\rangle^{*}$
%	such that $(v,v)\le s$, $(v,r)=1$, and $(v,u)\in\{0,1\}$ for all $u\in\aes$.
%\end{lemma}

\section{An integral overlattice of $A_5^3 \oplus A_1^4$}	\label{sec:A}

%We found sets of $57$ equiangular lines in dimension $18$ in some integral overlattice of $A_5^3 \oplus A_1^4$. 
In this section, we introduce notation for the vectors in
an integral overlattice of $A_5^3 \oplus A_1^4$,
and show that, up to switching equivalence, every affine equiangular set in that lattice
consists of vectors
%$4$ subsets of cardinality $36$, $7$, $7$ and $7$ 
in four cosets of $A_5^3 \oplus A_1^4$ in the overlattice.
After this section, we further investigate the vectors in these cosets in terms of combinatorial objects.
From now on, we set $\Lambda=A_5^3 \oplus A_1^4$.

\subsection{An integral overlattice $L$}

We provide a large number of affine equiangular sets of cardinality $57$ in an integral overlattice $L$ of $\Lambda$, whose discriminant is $6$.
The lattice $L$ is defined to be the integral overlattice of $\Lambda$ generated by
\begin{align*}
	\beta_0 & =(\alpha_5,\alpha_5,\alpha_5,0,0,0,\alpha_1) \in \R^{3 \cdot 6 + 4 \cdot 2},  \\
	\beta_1 & =(3\alpha_5,0,0,0,\alpha_1,\alpha_1,\alpha_1) \in \R^{3 \cdot 6 + 4 \cdot 2}, \\
	\beta_2 & =(0,3\alpha_5,0,\alpha_1,0,\alpha_1,\alpha_1) \in \R^{3 \cdot 6 + 4 \cdot 2}, \\
	\beta_3 & =(0,0,3\alpha_5,\alpha_1,\alpha_1,0,\alpha_1) \in \R^{3 \cdot 6 + 4 \cdot 2}.
\end{align*}
Here, note that $\beta_3$ is not necessary as a generator of $L$
since
\begin{align}\label{beta3}
	\beta_3 \equiv 3\beta_0 - \beta_1 - \beta_2 \pmod{\Lambda}.
\end{align}
We consider affine equiangular sets with norm $3$ with respect to the switching root
\[\sr :=(0,0,0,0,0,0,2\alpha_1).\]
Then such an affine equiangular set is contained in the set
\[X:=\{x\in L\mid (x,x)=3,\;(\sr ,x)=1\}.\]

\subsection{Cosets in $L/\Lambda$}	\label{subsec:cosets}

Next, we study the intersections of $X$ with the cosets in $L/\Lambda$.
First, we determine complete representatives of the cosets.
After that we consider a decomposition of affine equiangular sets.

\begin{lemma}	\label{lem:coset}
	The quotient group $L/\Lambda$ is isomorphic to the direct product of the cyclic subgroups generated by $\beta_0+\Lambda$, $\beta_1+\Lambda$ and $\beta_2+\Lambda$,
	which are of order $6$, $2$ and $2$, respectively.
\end{lemma}
\begin{proof}
	By the definitions of $\beta_0$, $\beta_1$ and $\beta_2$, the orders of the cosets are $6$, $2$ and $2$, respectively.
	Thus, it suffices to show that these cosets satisfy no nontrivial relations.
	Indeed, for integers $a,b,c$, suppose $a \beta_0 + b \beta_1 + c \beta_2 \in \Lambda$.
	We have
	\begin{align}\label{lem:coset:1}
		a \beta_0 + b \beta_1 + c \beta_2
		= \left( (a + 3b)\alpha_5, (a + 3c)\alpha_5, a\alpha_5, c\alpha_1, b\alpha_1, (b+c)\alpha_1, (a+b+c)\alpha_1 \right).
	\end{align}
	Since the third, fourth and fifth components are in $A_5$, $A_1$ and $A_1$, respectively, we have $a \equiv 0 \pmod 6$, $c \equiv 0 \pmod 2$ and $b \equiv 0 \pmod 2$.
	This means that the cosets of $\beta_0$, $\beta_1$ and $\beta_2$ have no nontrivial relations.
\end{proof}

From this lemma, we can determine the intersections of $X$ with the cosets in $L/\Lambda$.
We prepare some notation first.
For $i_1,i_2,i_3\in[6]$, let
\[x_{i_1,i_2,i_3}=(\alpha_5^{(i_1)},\alpha_5^{(i_2)},\alpha_5^{(i_3)},0,0,0,\alpha_1)
	\in \beta_0 + \Lambda.\]
For $I\in\binom{[6]}{3}$ and $\varepsilon_1,\varepsilon_2\in\{\pm1\}$, let
\begin{align*}
	x^{(1)}_{I,\varepsilon_1,\varepsilon_2} & =\left( \alpha_I,0,0,\ 0,\varepsilon_1\alpha_1,\varepsilon_2\alpha_1,\ \alpha_1 \right) \in \beta_1 + \Lambda, \\
	x^{(2)}_{I,\varepsilon_1,\varepsilon_2} & =\left( 0,\alpha_I,0,\ \varepsilon_2\alpha_1,0,\varepsilon_1\alpha_1,\ \alpha_1 \right) \in \beta_2 + \Lambda, \\
	x^{(3)}_{I,\varepsilon_1,\varepsilon_2} & =\left( 0,0,\alpha_I,\ \varepsilon_1\alpha_1,\varepsilon_2\alpha_1,0,\ \alpha_1 \right) \in \beta_3 + \Lambda,
\end{align*}
where $\alpha_I$ is defined in \eqref{A5mv_def}.

\begin{lemma}	\label{lem:cosetX}
	For $\beta + \Lambda \in L/\Lambda$, we have
	\begin{align}\label{lem:cosetX:2}
		X \cap (\beta+\Lambda) = \begin{cases}
			                         \{ x_{i_1,i_2,i_3} \mid i_1,i_2,i_3 \in [6] \}                                                                       & \text{ if } \beta+\Lambda = \beta_0 + \Lambda,                  \\
			                         \{ \sr - x_{i_1,i_2,i_3} \mid i_1,i_2,i_3 \in [6] \}                                                                 & \text{ if } \beta+\Lambda = -\beta_0 + \Lambda,                 \\
			                         \{ x^{(i)}_{I,\varepsilon_1,\varepsilon_2} \mid I \in \binom{[6]}{3},\; \varepsilon_1,\varepsilon_2 \in \{\pm 1\} \} & \text{ if } \beta+\Lambda = \beta_i + \Lambda, \quad (i=1,2,3), \\
			                         \emptyset                                                                                                            & \text{ otherwise. }
		                         \end{cases}
	\end{align}
\end{lemma}
\begin{proof}
	By Lemma~\ref{lem:coset}, we may assume
	\[\beta=a\beta_0+b\beta_1+c\beta_2,\]
	with
	\[a\in\{0,1,\dots,5\},\quad b,c\in\{0,1\}.\]
	Assume that $X\cap(\beta+\Lambda)\neq\emptyset$.
	Then the coset $\beta+\Lambda$ contains a representative
	\[x = (x_1\alpha_5,x_2\alpha_5,x_3\alpha_5,x_4\alpha_1,x_5\alpha_1,x_6\alpha_1,x_7\alpha_1),\]
	where, by \eqref{lem:coset:1}, the integers $x_1,\dots,x_7$ satisfy
	\begin{align}
		 & x_1\equiv a+3b,\quad x_2\equiv a+3c,\quad x_3\equiv a\pmod 6,\notag                         \\
		 & x_4\equiv c,\quad x_5\equiv b,\quad x_6\equiv b+c,\quad x_7=1\equiv a+b+c\pmod 2.\label{x7}
	\end{align}
	Here, we may choose $x_7=1$ since $(\rho,y)=1$ for every $y\in X$.

	Suppose there exists $y\in X\cap(\beta+\Lambda)$.
	Then
	\begin{equation}\label{yinx+Lambda}
		y=x+(\gamma_1,\gamma_2,\dots,\gamma_6,0)
	\end{equation}
	for some $\gamma_1,\gamma_2,\gamma_3\in A_5$ and
	$\gamma_4,\gamma_5,\gamma_6\in A_1$.
	Now
	\begin{align}
		3 & =(y,y)
		\nexteq
		\sum_{i=1}^3(x_i\alpha_5+\gamma_i,x_i\alpha_5+\gamma_i)+
		\sum_{i=4}^6(x_i\alpha_1+\gamma_i,x_i\alpha_1+\gamma_i)+\frac12
		\\&\geq
		%\min(\beta+\Lambda) 
		%\\&=\min(x+\Lambda)
		%\nexteq
		\sum_{i=1}^3\min(x_i\alpha_5+A_5)+
		\sum_{i=4}^6\min(x_i\alpha_1+A_1)+\frac12
		\label{3=}
		\nexteq
		\min((a+3b)\alpha_5+A_5)+\min((a+3c)\alpha_5+A_5)+\min(a\alpha_5+A_5)
		\\&\quad+
		\min(c\alpha_1+A_1)+\min(b\alpha_1+A_1)+\min((b+c)\alpha_1+A_1)+\frac12
		\nexteq
		\min((a+3b)\alpha_5+A_5)+\min((a+3c)\alpha_5+A_5)+\min(a\alpha_5+A_5)
		\\&\quad+
		\frac12(\delta_{c,1}+\delta_{b,1}+\delta_{b,1-c}+1)
		\nexteq
		\begin{cases}
			3\min(a\alpha_5+A_5)+\frac12                         & \text{if $(b,c)=(0,0)$,} \\
			2\min((a+3)\alpha_5+A_5)+\min(a\alpha_5+A_5)+\frac32 & \text{if $(b,c)=(1,1)$,} \\
			\min((a+3)\alpha_5+A_5)+2\min(a\alpha_5+A_5)+\frac32 & \text{otherwise.}
		\end{cases}
	\end{align}

	If $(b,c)=(0,0)$, then $\min(a\alpha_5+A_5)\leq\frac56$.
	By \eqref{A5min2} and \eqref{A5min3}, we have $a\notin\{2,3,4\}$.
	Then by \eqref{x7}, we have $a\neq0$, so $a\in\{1,5\}$, and
	equality holds in \eqref{3=}.

	If $(b,c)=(1,1)$, then
	\[\frac32\geq2\min((a+3)\alpha_5+A_5)+\min(a\alpha_5+A_5).\]
	By \eqref{A5min}, \eqref{A5min2} and \eqref{A5min3}, we obtain
	$a=3$, and
	equality holds in \eqref{3=}.

	If $(b,c)=(1,0)$ or $(0,1)$, then
	\[\frac32\geq\min((a+3)\alpha_5+A_5)+2\min(a\alpha_5+A_5).\]
	By \eqref{A5min}, \eqref{A5min2} and \eqref{A5min3}, we obtain
	$a=0$, and
	equality holds in \eqref{3=}.

	Therefore,
	\begin{align*}
		\beta+\Lambda & \in\{\beta_0+\Lambda,5\beta_0+\Lambda,
		3\beta_0+\beta_1+\beta_2+\Lambda,\beta_1+\Lambda,\beta_2+\Lambda\}
		\nexteq
		\{\beta_0+\Lambda,-\beta_0+\Lambda,
		\beta_3+\Lambda,\beta_1+\Lambda,\beta_2+\Lambda\}.
	\end{align*}

	Next, we determine the elements of $X\cap(\beta+\Lambda)$ for the above five cases.
	Since we have shown that equality holds in \eqref{3=} in all the cases,
	an arbitrary vector $y \in X\cap(\beta+\Lambda)$
	of the form \eqref{yinx+Lambda}
	must satisfy
	\begin{align*}
		(x_1\alpha_5+\gamma_1,x_1\alpha_5+\gamma_1) & =
		\min((a+3b)\alpha_5+A_5),                       \\
		(x_2\alpha_5+\gamma_2,x_2\alpha_5+\gamma_2) & =
		\min((a+3c)\alpha_5+A_5),                       \\
		(x_3\alpha_5+\gamma_3,x_3\alpha_5+\gamma_3) & =
		\min(a\alpha_5+A_5),                            \\
		(x_4\alpha_1+\gamma_4,x_4\alpha_1+\gamma_4) & =
		\min(c\alpha_1+A_1),                            \\
		(x_5\alpha_1+\gamma_5,x_5\alpha_1+\gamma_5) & =
		\min(b\alpha_1+A_1),                            \\
		(x_6\alpha_1+\gamma_6,x_6\alpha_1+\gamma_6) & =
		\min((b+c)\alpha_1+A_1).
	\end{align*}
	From \eqref{A5mv}, \eqref{A5mv3} and \eqref{A1mv},
	we see that the intersections $X\cap(\beta+\Lambda)$ are as given
	in the statement.
\end{proof}

By Lemma~\ref{lem:cosetX}, we have
\[
	|X\cap(\pm\beta_0+\Lambda)|=6^3=216,\qquad
	|X\cap(\beta_i+\Lambda)|=\binom{6}{3}\times4=80\quad(i\in\{1,2,3\}).
\]
Therefore, $X$ consists of $2\cdot216+3\cdot80=672$ vectors.
In the remainder of this subsection,
we investigate how an affine equiangular set decomposes into its intersection with the five cosets in Lemma~\ref{lem:cosetX}.

\begin{lemma}	\label{lem:crossinner}
	Let $0 \leq i < j \leq 3$.
	If $x \in X \cap (\beta_i + \Lambda)$ and $y \in X \cap (\beta_j + \Lambda)$,
	then $(x,y) \in \{0,1\}$.
\end{lemma}
\begin{proof}
	There exists $I \in \binom{[6]}{3}$ and $\varepsilon_1,\varepsilon_2 \in \{\pm 1\}$ such that $y = x^{(j)}_{I,\varepsilon_1,\varepsilon_2}$ by Lemma~\ref{lem:cosetX}.
	Assume $i \in \{1,2,3\}$.
	Then, by Lemma~\ref{lem:cosetX}, there exist $I', \varepsilon_1', \varepsilon_2'$ such that $x = x^{(i)}_{I',\varepsilon_1',\varepsilon_2'}$.
	We have
	\begin{align*}
		(x,y)
		= (x^{(i)}_{I',\varepsilon_1',\varepsilon_2'}, x^{(j)}_{I,\varepsilon_1,\varepsilon_2})
		\in \left\{ \frac{1}{2}(\varepsilon_1 \varepsilon_2'+1),\frac{1}{2}(\varepsilon_1' \varepsilon_2+1) \right\}\subseteq \{0,1\}.
	\end{align*}
	Next, assume that $i=0$.
	Then, by Lemma~\ref{lem:cosetX} again, we have $x=x_{i_1,i_2,i_3}$ for some $i_1,i_2,i_3 \in [6]$.
	Hence,
	\begin{align*}
		(x,y)
		= (x_{i_1,i_2,i_3}, x^{(j)}_{I,\varepsilon_1,\varepsilon_2})
		= (\alpha_5^{(i_j)}, \alpha_I) + \frac{1}{2}
		= (-e_{i_j}, \alpha_I) + \frac{1}{2}
		= \begin{cases}
			  1 & \text{if } i_j \in I, \\
			  0 & \text{otherwise.}
		  \end{cases}
	\end{align*}
	Therefore, we see that $(x,y) \in \{0,1\}$ as desired.
\end{proof}

Let $\cX$ be the set of affine equiangular sets with respect to the switching root $\sr$ in $X$.
For $i \in \{0,1,2,3\}$, let $\cX_i$ be the set of affine equiangular sets with respect to the switching root $\sr$ in $X \cap (\beta_i+\Lambda)$.
As shown in the following lemma, we may decompose affine equiangular sets.
Additionally, we do not need to consider the coset $-\beta_0 + \Lambda$ when we study $\cX$ up to switching equivalence.
Indeed, since $\rho\in\Lambda$, switching replaces each $x\in X\cap(-\beta_0+\Lambda)$ with $\rho-x\in X\cap(\beta_0+\Lambda)$.

\begin{lemma}	\label{lem:decompX}
	If $\omega \in \cX$,
	then $\omega$ is switching equivalent to the disjoint union $\bigsqcup_{i=0}^3 \omega_i$ for some $\omega_i \in \cX_i$ ($i=0,1,2,3$).
	Conversely,
	if $\omega_i \in \cX_i$ for $i=0,1,2,3$,
	then $\bigsqcup_{i=0}^3 \omega_i \in \cX$.
\end{lemma}
\begin{proof}
	For the first assertion, let $\omega \in \cX$.
	By Lemma~\ref{lem:cosetX}, we have
	\begin{align*}
		\omega = \left( \omega \cap (-\beta_0 + \Lambda) \right) \cup \bigcup_{i=0}^3 \left( \omega \cap (\beta_i + \Lambda) \right).
	\end{align*}
	We define
	\begin{align*}
		\omega_0 & := \left\{ \sr - x \mid x \in \omega \cap (-\beta_0 + \Lambda) \right\} \cup \left( \omega \cap (\beta_0 + \Lambda) \right) \subset X \cap (\beta_0 + \Lambda), \\
		\omega_i & := \omega \cap (\beta_i + \Lambda) \quad (i=1,2,3).
	\end{align*}
	Then, $\bigcup_{i=0}^3 \omega_i$ is switching equivalent to $\omega$ with respect to the switching root $\sr$.
	Hence, this is the desired decomposition.
	The second assertion follows from Lemma~\ref{lem:crossinner}.
\end{proof}

\section{Affine equiangular sets in the coset $\beta_0+\Lambda$}	\label{sec:beta0}
In this section,
we describe maximum affine equiangular sets in $\beta_0 + \Lambda$, by showing that they are in bijective correspondence with Latin squares of order $6$.

\begin{lemma}\label{lem:beta0}
	For $i_1, j_1, i_2, j_2, i_3, j_3 \in [6]$,
	\[(x_{i_1,i_2,i_3},x_{j_1,j_2,j_3})=
		\delta_{i_1,j_1}+\delta_{i_2,j_2}+\delta_{i_3,j_3}.\]
	Furthermore, the mapping
	\begin{align*}
		\begin{array}{ccc}
			\left\{ \text{independent sets in } H(3,6) \right\} & \to     & \cX_0                                                  \\
			C                                                   & \mapsto & \left\{ x_{i_1,i_2,i_3} : (i_1,i_2,i_3) \in C \right\}
		\end{array}
	\end{align*}
	is bijective.
\end{lemma}
\begin{proof}
	For $i,j\in[6]$,
	$(\alpha_5^{(i)},\alpha_5^{(j)})=\delta_{ij}-\frac16$.
	Hence,
	\begin{align*}
		(x_{i_1,i_2,i_3},x_{j_1,j_2,j_3})
		= & (\alpha_5^{(i_1)},\alpha_5^{(j_1)})+(\alpha_5^{(i_2)},\alpha_5^{(j_2)})+(\alpha_5^{(i_3)},\alpha_5^{(j_3)})+(\alpha_1,\alpha_1) \\
		= & \left(\delta_{i_1,j_1}-\frac16\right)+\left(\delta_{i_2,j_2}-\frac16\right)+\left(\delta_{i_3,j_3}-\frac16\right)+\frac12       \\
		= & \delta_{i_1,j_1}+\delta_{i_2,j_2}+\delta_{i_3,j_3}.
	\end{align*}
	This is the first assertion.
	Moreover, the second assertion follows immediately from the first one.
\end{proof}

\begin{lemma}	\label{lem:Hamming}
	The independence number of the Hamming graph $H(3,n)$ is $n^2$.
	Furthermore, the mapping
	\begin{align*}
		\begin{array}{ccc}
			\left\{ \latin : \text{Latin square of order } n \right\} & \to     & \left\{ \text{maximum independent sets in } H(3,n) \right\} \\
			\latin                                                    & \mapsto & \{(i,j,\latin_{ij}) : i,j \in [n]\}
		\end{array}
	\end{align*}
	is bijective.
\end{lemma}
\begin{proof}
	The vertex set of the Hamming graph $H(3,n)$ can be partitioned into $n^2$ cliques $\{(i,j,k) : k \in [n]\}$ of cardinality $n$ for $i, j \in [n]$.
	Hence, the independence number of $H(3,n)$ is at most $n^2$.
	Since there exists a Latin square $S$ of order $n$, $\{(i,j,\latin_{ij}) : i,j \in [n]\}$ is an independent set of size $n^2$ in $H(3,n)$.

	Next, it follows from the definition of a Latin square that the mapping is well-defined and injective.
	It remains to show that the mapping is surjective.
	Let $I$ be an arbitrary maximum independent set of $H(3,n)$.
	Since $I$ is an independent set, $I = \{(i,j,S_{ij}) : i,j \in [n]\}$ for some array $(S_{ij})_{i,j \in [n]}$,
	and this array must be a Latin square of order $n$.
	Therefore, the mapping is surjective.
\end{proof}

\begin{remark}
	Lemma~\ref{lem:Hamming} can be proved in a more general context.
	Indeed, Roos~\cite[Lemma 4.1.1]{MR687726} implies that an $n$-clique in $H(3,n)$ is an antidesign with dual diameter $2$.
	This implies that an independent set $Y$ of size $n^2$ in $H(3,n)$ achieves equality in~\cite[(6.4)]{MR498190}, so \cite[Theorem~6.2]{MR498190} implies that $Y$ is a $2$-design, that is, an orthogonal array of strength $2$.
	It is well-known and easy to see that an orthogonal array of strength $2$ and size $n^2$
	in $H(3,n)$ is equivalent to a Latin square of order $n$
	via the bijection described in Lemma~\ref{lem:Hamming}.
\end{remark}

\begin{theorem} \label{thm:beta0}
	The maximum cardinality of an affine equiangular set with respect to the switching root $\sr$ in
	$X\cap(\beta_0+\Lambda)$
	is $36$.
	Furthermore, the mapping
	\begin{align*}
		\begin{array}{rccc}
			\phi^{(0)}: & \left\{ \latin : \text{Latin square of order } 6 \right\} & \to     & \{ \omega \in \cX_0 : \omega \text{ is maximum in } \cX_0 \} \\
			            & \latin                                                    & \mapsto & \{ x_{i,j,\latin_{ij}} : i,j \in \{1,\ldots,6\} \}
		\end{array}
	\end{align*}
	is bijective.
	In addition, the mapping
	\begin{align*}
		\begin{array}{rccc}
			\bar{\phi}^{(0)}: & \left\{ \latin : \text{Latin square of order } 6 \right\} & \to     & \{ \omega \in \cX_{\pm 0} : \omega \text{ is maximum in } \cX_{\pm 0} \}/ \sim_{\sw} \\
			                  & \latin                                                    & \mapsto & [\{ x_{i,j,\latin_{ij}} : i,j \in \{1,\ldots,6\} \}],
		\end{array}
	\end{align*}
	is a bijection,
	where $\cX_{\pm 0}$ is the set of affine equiangular sets with respect to the switching root $\sr$ in $X \cap \left( (\beta_0+\Lambda) \cup (-\beta_0+\Lambda) \right)$.
\end{theorem}
\begin{proof}
	From Lemma~\ref{lem:beta0}, maximum affine equiangular sets in $\cX_0$ bijectively correspond to maximum independent sets in $H(3,6)$.
	Since the independence number of $H(3,6)$ is $36$ by Lemma~\ref{lem:Hamming}, the maximum cardinality of an affine equiangular set in $\cX_0$ is $36$.
	Furthermore, from the bijection in Lemma~\ref{lem:Hamming}, we see that the mapping $\phi^{(0)}$ is bijective.

	Next, the switching class of $\omega \in \cX_{\pm 0}$ is written as
	\begin{align*}
		[\omega] = \{ \theta \cup \{ \rho - x : x \in \omega \setminus \theta \} : \theta \subset \omega \}.
	\end{align*}
	Since either $x$ or $\sr - x$ is in $\beta_0 + \Lambda$ for each $x \in X \cap \left( (\beta_0+\Lambda) \cup (-\beta_0+\Lambda) \right)$, we see that $\cX_0$ is a complete set of representatives of $\cX_{\pm 0}/\sim_{\sw}$.
	Therefore, the mapping $\bar{\phi}^{(0)}$ is bijective.
\end{proof}

\section{Affine equiangular sets in $\beta_i + \Lambda$ ($i=1,2,3$)}	\label{sec:beta123}
In this section, we describe affine equiangular sets in $\beta_i + \Lambda$ for $i=1,2,3$, that is, elements of $\cX_i$.
The Johnson graph $J(v,k)$ is defined to be the graph whose vertex set is $\binom{[v]}{k}$, and two vertices are adjacent if their intersection has size $k-1$.
Also, we regard $K_2$ as the complete graph with vertex set $\{\pm 1\}$.
Recall that the Hamming graph $H(2,2)$ is the cartesian product $K_2 \boxempty K_2$.
To study affine equiangular sets in $\beta_i + \Lambda$ ($i=1,2,3$), we consider the distance $2$-or-$3$ graph, denoted by $H$, of the cartesian product $J(6,3) \boxempty H(2,2)$,
and the distance $2$-or-$3$ graph, denoted by $K$, of the cartesian product $J(6,3) \boxempty K_2$.

\begin{lemma}\label{lem:J63H22}
	Let $i\in[3]$.
	Then, the mapping
	\begin{align*}
		\begin{array}{cccc}
			\phi^{(i)} : & \left\{ \text{cliques } \text{ in } H \right\} & \to     & \cX_i,                                                                                  \\
			             & C                                              & \mapsto & \{ x^{(i)}_{I,\varepsilon_1,\varepsilon_2} : (I,(\varepsilon_1,\varepsilon_2)) \in C \}
		\end{array}
	\end{align*}
	is bijective.
	In addition,
	\begin{align*}
		\begin{array}{cccc}
			\bar{\phi}^{(i)}: & \left\{ \text{cliques } \text{ in } K \right\} & \to     & \cX_i/\sim_{\sw},                                         \\
			                  & C                                              & \mapsto & [\{ x^{(i)}_{I,\varepsilon,1} : (I,\varepsilon) \in C \}]
		\end{array}
	\end{align*}
	is bijective.
\end{lemma}
\begin{proof}
	By Lemma~\ref{lem:cosetX}, the mapping
	\begin{align*}
		\begin{array}{cccc}
			 & \binom{[6]}{3} \times \{\pm 1\}^2 & \to     & X \cap (\beta_i+\Lambda)                \\
			 & (I,(\varepsilon_1,\varepsilon_2)) & \mapsto & x^{(i)}_{I,\varepsilon_1,\varepsilon_2}
		\end{array}
	\end{align*}
	is bijective.
	Hence, it suffices to show that the adjacency is equivalent to two vectors having inner product $0$ or $1$.
	Indeed, for
	$I,I'\in\binom{[6]}{3}$ and $\varepsilon_1,\varepsilon_1',\varepsilon_2,\varepsilon_2'\in\{\pm1\}$,
	we have
	\begin{align*}
		(x^{(i)}_{I,\varepsilon_1,\varepsilon_2},x^{(i)}_{I',\varepsilon_1',\varepsilon_2'})
		 & =|I\cap I'|-\frac32+\frac12(\varepsilon_1\varepsilon_1'+\varepsilon_2\varepsilon_2'+1)
		%\nexteq
		%|I\cap I'|+\delta_{\varepsilon_1,\varepsilon_1'}+\delta_{\varepsilon_2,\varepsilon_2'}
		\nexteq
		3-\left(d_J(I,I')+d_H((\varepsilon_1,\varepsilon_2),(\varepsilon_1',\varepsilon_2'))\right),
		\nexteq
		3 - d\big( (I,(\varepsilon_1,\varepsilon_2)), (I',(\varepsilon_1',\varepsilon_2')) \big),
	\end{align*}
	where $d_J$, $d_H$ and $d$ are the distance functions of $J(6,3)$, $H(2,2)$ and $J(6,3) \boxempty H(2,2)$, respectively.

	Next, the complete graph $K_2$ is an induced subgraph of $H(2,2)$ through the inclusion $\iota: \{\pm 1\} \to \{\pm 1\}^2$ defined by $\iota(\varepsilon) = (\varepsilon,1)$ for $\varepsilon \in \{\pm 1\}$.
	This allows us to regard $J(6,3) \boxempty K_2$ as an induced subgraph of $J(6,3) \boxempty H(2,2)$.
	Noting that the restriction of $d$ to $J(6,3) \boxempty K_2$ coincides with the distance function of $J(6,3) \boxempty K_2$,
	we see that $K$ is an induced subgraph of $H$.
	Hence, cliques in $K$ are also cliques in $H$, and the mapping $\bar{\phi}^{(i)}$ is well-defined.
	To prove bijectivity, we show that every switching class in $\cX_i/\sim_{\sw}$ has a unique representative consisting of vectors $x^{(i)}_{I,\varepsilon,1}$ for $I \in \binom{[6]}{3}$ and $\varepsilon \in \{\pm 1\}$.
	Indeed, fix a switching class $[\omega] \in \cX_i/\sim_{\sw}$.
	This can be written as
	\begin{align*}
		[\omega] = \big\{ \theta \cup \left\{ \rho-x : x \in \omega \setminus \theta \right\} : \theta \subset \omega \big\}.
	\end{align*}
	Noting that
	\begin{align*}
		\rho - x^{(i)}_{I,\varepsilon_1,\varepsilon_2} = x^{(i)}_{[6] \setminus I,-\varepsilon_1,-\varepsilon_2}\quad (I \in \binom{[6]}{3},\; \varepsilon_1,\varepsilon_2 \in \{\pm 1\}),
	\end{align*}
	we see that $\omega$ has a unique representative consisting of vectors $x^{(i)}_{I,\varepsilon,1}$ for $I \in \binom{[6]}{3}$ and $\varepsilon \in \{\pm 1\}$.
	This means that $\bar{\phi}^{(i)}$ is bijective.
\end{proof}

We explicitly determine the maximum affine equiangular sets in
$X \cap (\beta_i + \Lambda)$.
They correspond to a Pasch configuration $\{1,3,5\}$, $\{1,4,6\}$, $\{2,3,6\}$, $\{2,4,5\}$ and three vertices up to some action.
To this end, we consider cliques in the graph $K$,
which correspond to switching classes of affine equiangular sets in $X \cap (\beta_i + \Lambda)$ by Lemma~\ref{lem:J63H22}.
Since $S_6$ and $\Sym(\{\pm1\})$ act on $J(6,3)$ and $K_2$, respectively,
we consider the cliques up to the natural action of the subgroup $S_6 \times \Sym(\{\pm1\}) \subset \Aut(K)$.
Here, $\Sym(X)$ is the symmetric group on a set $X$.

\begin{proposition}	\label{prop:J63K27clique}
	The graph $K$ has clique number $7$.
	A complete system of representatives of the action of $S_6 \times \Sym(\{\pm 1\})$ on the set of maximum cliques in $K$ is given by $c_1$ and $c_2$, each consisting of
	\begin{align}\label{prop:J63K27clique:1}
		(\{1,3,5\},-1),\ (\{1,4,6\},-1),\ (\{2,3,6\},-1),\ (\{2,4,5\},-1),
	\end{align}
	and either the three vertices
	\begin{align}\label{prop:J63K27clique:2}
		(\{1,2\} \cup \{3\},1),\ (\{3,4\} \cup \{5\},1),\ (\{5,6\} \cup \{1\},1),
	\end{align}
	or
	\begin{align}\label{prop:J63K27clique:3}
		(\{1,2\} \cup \{6\},1),\ (\{3,4\} \cup \{2\},1),\ (\{5,6\} \cup \{4\},1).
	\end{align}
	In particular, a complete system of representatives of the action of $S_6$ on the set of maximum cliques of size $7$ in $K$ is given by $\pm c_1$ and $\pm c_2$.
	Here, $-c_i := \{ (I,-\varepsilon) : (I,\varepsilon) \in c_i \}$ for $i=1,2$.
\end{proposition}
\begin{proof}
	Using a computer, we verified that the clique number of $K$ is $7$,
	and enumerated all $960$ cliques of size $7$ in $K$.
	Moreover, using a computer, we may verify that two orbits of the action of $S_6 \times \Sym(\{\pm 1\})$ of cardinality $480$ each, are represented by $c_1$ and $c_2$.
	Similarly, we may verify that four orbits of the action of $S_6$ of cardinality $240$ each, are represented by $\pm c_1$ and $\pm c_2$.
\end{proof}

\begin{remark}
	In the proof of Proposition~\ref{prop:J63K27clique}, one may replace $S_6 \times \Sym(\{\pm 1\})$ with $\Aut(K)$ when considering orbits.
	However, the resulting orbit decomposition is the same and does not simplify the classification.
\end{remark}

\section{Affine equiangular sets in $L$}	\label{sec:main}
From the previous results, we can determine the maximum affine equiangular sets in $L$ as stated in the following theorem.
In the theorem, we use the notation for vectors in $L$ introduced in Subsection~\ref{subsec:cosets}.
After the theorem, we describe the natural actions of $S_6^3$ on these maximum affine equiangular sets, and the Seidel matrices of these sets.

\begin{theorem} \label{thm:main}
	The maximum cardinality of an affine equiangular set with respect to the switching root $\sr$ in $L$ is $57$.
	For every such maximum affine equiangular set $\omega$, there exist a unique Latin square $S$ of order $6$ and unique cliques $C_1, C_2, C_3$ of size $7$ in the distance $2$-or-$3$ graph of the cartesian product $J(6,3) \boxempty K_2$ such that $\omega$ is switching equivalent to $\omega_0 \cup \omega_1 \cup \omega_2 \cup \omega_3$ where
	\begin{align}
		\omega_0 & = \left\{ x_{i,j,S_{ij}} : i,j \in [6] \right\} \label{thm:main:omega0}
	\end{align}
	and
	\begin{align}
		\omega_i & = \left\{ x^{(i)}_{I,\varepsilon,1} : (I,\varepsilon) \in C_i \right\} \label{thm:main:omegai}
	\end{align}
	for $i=1,2,3$.
\end{theorem}
\begin{proof}
	Let $[\omega] \in \cX/\sim_{\sw}$ be a switching class of maximum affine equiangular sets in $L$.
	From Lemma~\ref{lem:decompX}, there exist $\omega_i \in \cX_i$ ($i=0,1,2,3$) such that
	\begin{align*}
		[\omega] = [\omega_0 \cup \omega_1 \cup \omega_2 \cup \omega_3].
	\end{align*}
	Since $\omega$ is maximum, the second assertion of Lemma~\ref{lem:decompX} implies that $\omega_i$ is maximum in $\cX_i$ for each $i\in\{0,1,2,3\}$.
	By Theorem~\ref{thm:beta0}, the set $\omega_0$ is given as in the statement.
	Also, by Lemma~\ref{lem:J63H22} and Proposition~\ref{prop:J63K27clique}, each $\omega_i$ ($i=1,2,3$) is given as in the statement.
	In particular, the cardinality of $\omega$ is $36 + 3 \cdot 7 = 57$.
\end{proof}

\begin{remark}	\label{remark:main}
	Let $C$ be the set of four vertices in~\eqref{prop:J63K27clique:1}.
	Define $\omega_0$ as in~\eqref{thm:main:omega0} for an arbitrary Latin square $S$, and define $\omega_i$ as in~\eqref{thm:main:omegai} with $C_i=C$ for $i=1,2,3$.
	Then $\omega_0\cup\omega_1\cup\omega_2\cup\omega_3$ consists of $48$ vectors in $X$.
	These vectors, together with the switching root, span a subspace of codimension $1$ in $\R L \simeq \R^{19}$.
	This gives a set of $48$ equiangular lines in $\R^{17}$, known to be maximum
	(see Table~\ref{table:N(d)}).
\end{remark}

Next, we describe a natural action of $S_6^3$ on the set of maximum affine equiangular sets in $L$.
Indeed, we identify the group $S_6^3$ with the direct product of $\Sym(\{1,\ldots,6\})$, $\Sym(\{7,\ldots,12\})$ and $\Sym(\{13,\ldots,18\})$.
Then, we see by the definition of $L$ that $S_6^3$ acts on the set
\begin{align*}
	\cM & = \{ \omega \in \cX : \omega \text{ is maximum in } \cX \}
\end{align*}
%of switching classes of maximum affine equiangular sets in $L \subset \R^{26}$ 
by permuting the first $18$ coordinates of $\R^{26}$ containing $L$.

Let $\cL$ be the set of Latin squares of order $6$.
Recall that the group $S_6^3$ acts on $\cL$; see after Definition~\ref{dfn:latin}.
Let $\cC$ be the set of maximum cliques in $K$, which is the distance $2$-or-$3$ graph of $J(6,3) \boxempty K_2$.
Since $S_6 \subset \Aut(J(6,3))$ acts naturally on $K$, the group $S_6^3$ acts on $\cC^3$ by acting on each coordinate.
Hence, $S_6^3$ acts naturally on the set $\cL \times \cC^3$.
Furthermore, this action of $S_6^3$ is compatible with the action on $\cM$ via the following mapping:
\begin{align}\label{eq:f}
	\begin{array}{rccc}
		f : & \cL \times \cC^3        & \to     & \cM                                                \\
		    & (\latin, C_1, C_2, C_3) & \mapsto & \omega_0 \cup \omega_1 \cup \omega_2 \cup \omega_3
	\end{array}
\end{align}
where $\omega_0$ is defined as in~\eqref{thm:main:omega0} and $\omega_i$ is defined as in~\eqref{thm:main:omegai} for $i=1,2,3$.
Hence, the function $f$ induces a surjection
\begin{align}\label{eq:bar_f}
	\begin{array}{rccc}
		\bar{f} : & \left( \cL \times \cC^3 \right) / S_6^3 & \to     & f(\cL \times \cC^3) / S_6^3     \\
		          & S_6^3 (\latin, C_1, C_2, C_3)           & \mapsto & S_6^3 f(\latin, C_1, C_2, C_3).
	\end{array}
\end{align}

As we will see in the proof of Corollary~\ref{cor:characteristic} below,
the characteristic polynomial of the Seidel matrix corresponding to a maximum affine equiangular set in $L$ is invariant under the action of $S_6^3$.
This fact allows us to show that there are only two possible characteristic polynomials
(see Corollary~\ref{cor:characteristic}).
Before stating the corollary, we recall the definition of the Seidel matrix corresponding to an affine equiangular set.
For an affine equiangular set $\omega = \{u_1, \ldots, u_n\}$ with norm $3$ with respect to the switching root $r$, the corresponding Seidel matrix is defined to be the matrix $B^\top B - 5I$ where
\begin{align}
	B := \begin{bmatrix}
		     \widetilde{u_1}^\top & \widetilde{u_2}^\top & \cdots & \widetilde{u_{n}}^\top
	     \end{bmatrix}
\end{align}
Here, the vectors $\widetilde{u_i}$ are defined in~\eqref{eq:tilde_def}.

\begin{corollary}	\label{cor:characteristic}
	The characteristic polynomial of the Seidel matrix corresponding to a maximum affine equiangular set in $L$ is either
	\begin{align*}
		(x + 5)^{39} (x - 13)^8 (x - 9)^8 (x^2 - 19x + 72)
	\end{align*}
	or
	\begin{align*}
		(x + 5)^{39} (x - 13)^6 (x - 9)^6 (x - 11) (x-12) (x^2 - 20x + 87)^2.
	\end{align*}
\end{corollary}
\begin{proof}
	Let $\omega$ be a maximum affine equiangular set in $L$ with respect to the switching root $\sr$.
	By Theorem~\ref{thm:main}, we have
	\[\omega \sim_{\sw} \omega_0 \cup \omega_1 \cup \omega_2 \cup \omega_3\]
	where
	$\omega_0$ is defined as in~\eqref{thm:main:omega0} and $\omega_i$ is defined as in~\eqref{thm:main:omegai} for $i=1,2,3$ for some $(S, C_1, C_2, C_3) \in \cL \times \cC^3$.
	Since the characteristic polynomial of the Seidel matrix corresponding to $\omega$ is invariant under switching, we may assume that $\omega = \omega_0 \cup \omega_1 \cup \omega_2 \cup \omega_3$.

	Write $u_1, u_2, \ldots, u_{57}$ for the vectors in $\omega$.
	The corresponding Seidel matrix is
	$
		B^\top B - 5I
	$
	where
	$$
		B := \begin{bmatrix}
			\widetilde{u_1}^\top & \widetilde{u_2}^\top & \cdots & \widetilde{u_{57}}^\top
		\end{bmatrix}
		.$$
	Here, the vectors $\widetilde{u_i}$ are defined in~\eqref{eq:tilde_def} where $r$ is replaced with $\rho$.
	Since $B^\top B$ has the same nonzero eigenvalues as $B B^\top$, it suffices to determine the characteristic polynomial of
	$$B B^\top = \sum_{u \in \omega} \tilde{u}^\top \tilde{u}.$$
	Clearly, the characteristic polynomial of $B B^\top$ is invariant under the action of $S_6^3$.
	By the compatibility of the mapping $f$ defined in~\eqref{eq:f}, we may assume that $(C_1, C_2, C_3)$ is one of the $64$ representatives $(\pm c_{i_1}, \pm c_{i_2}, \pm c_{i_3})$ for $S_6^3$-orbits on $\cC^3$ (see Proposition~\ref{prop:J63K27clique}).

	Let $J$ be the all-ones matrix.
	Let $\one$ be the all-ones row vector, and $e_i$ be the row vector with $1$ in the $i$-th coordinate and $0$ elsewhere.
	We have by definition,
	\begin{align*}
		(\alpha_5^{(i)})^\top \alpha_5^{(j)}
		= \frac{1}{36} J + e_i^\top e_j - \frac{1}{6} e_i^\top \one - \frac{1}{6} \one^\top e_j.
	\end{align*}
	From this, we see that
	\begin{align*}
		\sum_{u \in \omega_0} \tilde{u}^\top \tilde{u}
		 & = \sum_{i,j=1}^6 \widetilde{x_{i,j,S_{ij}}}^\top \widetilde{x_{i,j,S_{ij}}}        \\
		 & = \sum_{i,j=1}^6
		\begin{bmatrix}
			\sqrt{2} \alpha_5^{(i)} & \sqrt{2}\alpha_5^{(j)} & \sqrt{2}\alpha_5^{(S_{ij})} & O
		\end{bmatrix}^\top
		\begin{bmatrix}
			\sqrt{2} \alpha_5^{(i)} & \sqrt{2} \alpha_5^{(j)} & \sqrt{2} \alpha_5^{(S_{ij})} & O
		\end{bmatrix} \\
		 & = \left[
			     \begin{array}{cccc}
				     12I_6-2J_6 & O          & O          & O   \\
				     O          & 12I_6-2J_6 & O          & O   \\
				     O          & O          & 12I_6-2J_6 & O   \\
				     O          & O          & O          & O_8
			     \end{array}
			     \right].
	\end{align*}
	This matrix is invariant under the action of $S_6^3$.
	Hence, $B B^\top$ can be computed by adding $\sum_{i=1}^3 \sum_{u \in \omega_i} \tilde{u}^\top \tilde{u}$ to the above matrix.
	By direct calculation, we see that the characteristic polynomial of $B B^\top - 5I$
	is either
	\begin{align*}
		(x + 5)^{8} (x - 13)^8 (x - 9)^8 (x^2 - 19x + 72)
	\end{align*}
	or
	\begin{align*}
		(x + 5)^{8} (x - 13)^6 (x - 9)^6 (x - 11) (x-12) (x^2 - 20x + 87)^2.
	\end{align*}
	Noting that the characteristic polynomial of Seidel matrix $B^\top B - 5I$ is $(x+5)^{57-26}$ times that of $B B^\top - 5I$, we obtain the desired conclusion.
\end{proof}

\section{Maximality in $\R L$ of affine equiangular sets in $L$}	\label{sec:maximality}

Although Theorem~\ref{thm:main} shows that there exist affine equiangular sets in $L$ of cardinality $57$,
which is the current lower bound for $N(18)$,
these sets are already maximal in dimension $18$,
as will be shown in Theorem~\ref{thm:maximality}.
Since there are too many such affine equiangular sets,
it is infeasible to verify the maximality of each of them.
Instead,
in Subsection~\ref{subsec:1},
we classify the possible lattices generated by affine equiangular sets in $L$ together with the switching root $\sr$.
Since the candidates for vectors that can be added to an affine equiangular set in $L$ are contained in the dual lattices of these possible lattices,
we prove in Subsection~\ref{subsec:2} that these affine equiangular sets are maximal.

\subsection{Lattices generated by affine equiangular sets with a switching root}	\label{subsec:1}

Before classifying the possible lattices generated by affine equiangular sets in $L$ together with the switching root $\sr$, we prepare some vectors and lattices.
Let $c_5 := \left( 1,1,0,0,-1,-1\right) \in A_5$ and
\[
	c := \left( c_5 , c_5, c_5, 2\alpha_1, 2\alpha_1, 2\alpha_1, 0 \right) \in L^*.
\]
Define
$L_{54} := \{ x \in L \mid (x,c) \equiv 0 \pmod 3 \}$,
which is an integral lattice of discriminant $54$.
Similarly, let
\[
	d := \left( 0 , \alpha_{\{4,5,6\}}, \alpha_{\{4,5,6\}}, \alpha_1, 0, 0, 0 \right) \in L^*.
\]
Then, define
$L_{24} := \{ x \in L \mid (x,d) \equiv 0 \pmod 2 \}$,
which is an integral lattice of discriminant $24$.
We remark that it can be verified that the root sublattices of $L_{54}$ and $L_{24}$ are $A_1^{10}$ and $A_5 \oplus A_2^4 \oplus A_1^3$, respectively.
In the following two lemmas, we investigate $L_{54}$ and $L_{24}$.

\begin{lemma}	\label{lem:overlattice}
	Every proper integral overlattice of $L_{54}$ or $L_{24}$ is isomorphic to $L$, via an isometry fixing the switching root $\sr$.
\end{lemma}
\begin{proof}
	Let $M$ be a proper integral overlattice of $L_{54}$.
	Then, $M/L_{54}$ is a nontrivial subgroup of $L_{54}^*/L_{54}$
	such that $|M/L_{54}||L_{54}^*/M^*|=|M/L_{54}|^2$ divides $|L_{54}^*/L_{54}|$.
	Using a computer, we can verify that $L_{54}^*/L_{54}$ is
	isomorphic to $\Z/2\Z \times (\Z/3\Z)^3$.
	Thus
	%	Also, it is known that 
	%	\begin{align*}
	%		\disc L_{54} / |M/L_{54}|^2 = \disc M \in \Z
	%	\end{align*}
	%	holds, and hence 
	$|M/L_{54}| = 3$, and $M$ is generated by $L_{54}$ together with a representative of a coset
	%	Therefore, every proper overlattice of $L_{54}$ is obtained from $L_{54}$ and an element 
	of order $3$ in the quotient group $L_{54}^*/L_{54}$.
	Among the $13$ subgroups of order $3$ in $L_{54}^*/L_{54}$, only $4$ subgroups correspond to
	an integral overlattice of $L_{54}$, all of which
	%	By checking which of these overlattices are integral, we can enumerate all proper integral overlattices of $L_{54}$.
	%	Furthermore, we can find that they
	are isometric to $L$ via an isometry fixing $\sr$.
	%	Similarly, we can verify that the desired result for $L_{24}$ with a computer.

	The proof for proper integral overlattices of $L_{24}$ is analogous, hence omitted.
\end{proof}

\begin{lemma}	\label{lem:noL24}
	No affine equiangular set in $L$ of cardinality $57$ together with the switching root $\sr$ generates the lattice $L_{24}$.
\end{lemma}
\begin{proof}
	Assume that there exists an affine equiangular set $\omega$ in $L$ of cardinality $57$ such that $\langle \omega \cup \{\sr\} \rangle = L_{24}$.
	By Theorem~\ref{thm:main}, the set $\omega$ is the disjoint union of four affine equiangular sets $\omega_i \in \cX_i$ ($i=0,1,2,3$) as described there.
	In particular, $\omega_0$ is given by~\eqref{thm:main:omega0} for some Latin square $\latin$ of order $6$.
	Recall that $d = ( 0 , \alpha_{\{4,5,6\}}, \alpha_{\{4,5,6\}}, \alpha_1, 0, 0, 0 )$ and $\alpha_{\{4,5,6\}} = \left( 1,1,1,-1,-1,-1\right)/2$.
	Since $\omega$ is contained in $L_{24} = \{ x \in L \mid (x,d) \equiv 0 \pmod 2 \}$, we have for $x_{i, j, S_{i,j}} \in \omega_0$ that
	\begin{align*}
		0
		 & \equiv (x_{i,j,S_{i,j}}, d) \pmod 2                                                             \\
		 & = 0 + ( \alpha_5^{(j)}, \alpha_{\{4,5,6\}} ) + ( \alpha_5^{(S_{i,j})}, \alpha_{\{4,5,6\}} ) + 0 \\
		 & = -( e_{j}, \alpha_{\{4,5,6\}} ) - ( e_{S_{i,j}}, \alpha_{\{4,5,6\}} ).
	\end{align*}
	In particular, the corresponding Latin square $S$ satisfies that
	\begin{align*}
		\latin_{i, 1} \in \{4,5,6\} \text{ for every } i \in [6].
	\end{align*}
	However, this is impossible since each symbol appears exactly once in the first column of $S$.
\end{proof}

From the above claims, we classify the possible lattices generated by maximum affine equiangular sets in $L$ together with the switching root $\sr$.

\begin{proposition}	\label{prop:classifylattice}
	Every maximum affine equiangular set in $L$ together with the switching root $\sr$ generates $L$ or $L_{54}$ up to isometry fixing $\sr$.
\end{proposition}
\begin{proof}
	By Theorem~\ref{thm:main}, the subset $f(\cL \times \cC^3)$ of $\cM$ is a complete set of representatives of $\cM / \sim_{\sw}$.
	Here, the mapping $f : \cL \times \cC^3 \to \cM$ is defined in~\eqref{eq:f}.
	If two affine equiangular sets $\omega$ and $\omega'$ with respect to a switching root $\sr$ are switching equivalent,
	then the two lattices $\langle \omega \cup \{\sr\} \rangle$ and $\langle \omega' \cup \{\sr\} \rangle$ are the same.
	Hence, it suffices to consider the lattices $\langle \omega \cup \{\sr\} \rangle$ for each $\omega \in f(\cL \times \cC^3)$.

	Also, we see that if two affine equiangular sets $\omega$ and $\omega'$ in $f(\cL \times \cC^3)$ are conjugate under the action of $S_6^3$,
	then the two lattices $\langle \omega \cup \{\sr\} \rangle$ and $\langle \omega' \cup \{\sr\} \rangle$ are isometric via an isometry fixing $\sr$.
	As discussed after Remark~\ref{remark:main}, the mapping $f$ induces a mapping $\bar{f} : \left( \cL \times \cC^3 \right) / S_6^3 \to f(\cL \times \cC^3) / S_6^3$.
	%Indeed, if there exists an element $\sigma \in S_6^3$ such that $[\omega'] = \sigma([\omega])$,
	%then $\sigma$ induces an isometry from $\langle \omega \cup \{\sr\} \rangle$ to $\langle \omega' \cup \{\sr\} \rangle$ fixing $\sr$ as desired.
	Therefore, it suffices to prove that for some representative $(\latin, C_1, C_2, C_3)$ of each orbit in $\left( \cL \times \cC^3 \right) / S_6^3$,
	\begin{align}\label{prop:classifylattice:0}
		\langle f(\latin, C_1, C_2, C_3) \cup \{\sr\} \rangle \in \left\{ L, L_{54}, L_{24} \right\},
	\end{align}
	up to isometry fixing $\sr$.
	To this end, we provide explicit representatives of the orbits in $\left( \cL \times \cC^3 \right) / S_6^3$.
	By Proposition~\ref{prop:J63K27clique},
	we have $\{c_1, -c_1, c_2, -c_2 \}$ as a set of representatives of $\cC/S_6$.
	Therefore, the set of tuples $(\latin,\pm c_i,\pm c_j,\pm c_k)$ with $\latin \in \cL$ and $i,j,k \in \{1,2\}$
	contains a complete set of representatives for
	$\left( \cL \times \cC^3 \right) / S_6^3$.
	Hence, our goal is reduced to proving that
	\begin{align}\label{prop:classifylattice:1}
		\langle f(\latin, \pm c_i, \pm c_j, \pm c_k) \cup \{\sr\} \rangle \in \left\{ L, L_{54}, L_{24} \right\}
	\end{align}
	for any $\latin \in \cL$ and $i,j,k=1,2$, up to isometry fixing $\sr$.
	However, it is difficult to directly verify this since the cardinality of $\cL$ is very large~\cite{MR2176596}. %frolov1890, 
	Therefore, we instead consider sublattices of these lattices.
	The set $\phi^{(0)}(\latin)$ contains the vectors
	$
		x_{1,1,\latin_{11}}, x_{1,2,\latin_{12}}, x_{2,1,\latin_{21}}, x_{2,2,\latin_{22}},
	$
	which are induced by the upper-left $2 \times 2$ subarray of $\latin$.
	Hence,
	\begin{align}\label{prop:classifylattice:2}
		\langle f(\latin, \pm c_i, \pm c_j, \pm c_k) \cup \{\sr\} \rangle
		= \langle \phi^{(0)}(\latin) \cup \phi^{(1)}(\pm c_i) \cup \phi^{(2)}(\pm c_j) \cup \phi^{(3)}(\pm c_k) \cup \{\sr\} \rangle
	\end{align}
	contains the sublattice
	\begin{align}
		\langle
		\{x_{1,1,s_{11}}, x_{1,2,s_{12}}, x_{2,1,s_{21}}, x_{2,2,s_{22}}\}
		\cup \phi^{(1)}(\pm c_i) \cup \phi^{(2)}(\pm c_j) \cup \phi^{(3)}(\pm c_k) \cup \{\sr\} \rangle \label{prop:classifylattice:4}
	\end{align}
	for some $s_{11}, s_{12}, s_{21}, s_{22} \in [6]$ with $s_{21} \neq s_{11} \neq s_{12}$ and $s_{21} \neq s_{22} \neq s_{12}$.
	It can be readily verified by computer that every lattice of the form~\eqref{prop:classifylattice:4}
	is isometric to $L$, $L_{54}$ or $L_{24}$ via an isometry fixing $\sr$.
	By Lemma~\ref{lem:overlattice}, all lattices of the form~\eqref{prop:classifylattice:2} are also isometric to $L$, $L_{54}$ or $L_{24}$ via an isometry fixing $\sr$.
	By Lemma~\ref{lem:noL24}, we conclude that~\eqref{prop:classifylattice:1} holds.
\end{proof}

This proposition shows the following.

\begin{theorem}	\label{thm:L}
	Every maximum affine equiangular set with respect to the switching root $\sr$ in $L$ is isometric to one with respect to the same switching root $\sr$ in $L$
	which together with $\sr$ generates either $L$ or $L_{54}$.
\end{theorem}
\begin{proof}
	Let $\omega$ be a maximum affine equiangular set in $L$ with respect to the switching root $\sr$.
	By Proposition~\ref{prop:classifylattice}, the lattice $\langle \omega \cup \{\sr\} \rangle$ is isometric to $L$ or $L_{54}$ via an isometry $f$ fixing $\sr$.
	Then, $\omega$ is isometric to $f(\omega)$ with the same switching root $\sr$, and $\langle f(\omega) \cup \{\sr\} \rangle$ coincides with $L$ or $L_{54}$.
	Therefore, $f(\omega)$ is the desired maximum affine equiangular set in $L$.
\end{proof}

\subsection{Maximality in $\R L$ of maximum affine equiangular sets in $L$}	\label{subsec:2}
In this subsection, we show that all the maximum affine equiangular sets in $L$ are maximal in $\R L$.

\begin{theorem}
	\label{thm:maximality}
	Every maximum affine equiangular set in $L$ with respect to the switching root $\sr$ is maximal in $\R L$.
\end{theorem}
\begin{proof}
	Let $\omega$ be a maximum affine equiangular set in $L$ with respect to the switching root $\sr$.
	By Theorem~\ref{thm:L}, this set $\omega$ is isometric to one which together with $\sr$ generates either $L$ or $L_{54}$.
	Without loss of generality, we may assume $M := \langle \omega \cup \{ \sr \} \rangle$ is either $L$ or $L_{54}$.

	For a contradiction, suppose that $\omega$ is not maximal in $\R L$.
	Then, there exists a vector $x \in \R L$ such that $\omega \cup \{ x \}$ is an affine equiangular set of cardinality $58$ with respect to the switching root $\sr$.
	Then, $N := \langle \omega \cup \{ x, \sr \} \rangle$ is an integral lattice containing $M$.
	In particular, $N$ contains $L_{54}$.
	Then by Lemma~\ref{lem:overlattice}, we have $N$ is contained in $L$ up to isometry fixing $\sr$.
	However, Theorem~\ref{thm:main} implies that every affine equiangular set in $L$ with respect to $\sr$ has cardinality at most $57$.
	This is a contradiction.
	Therefore, $\omega$ is maximal in $\R L$.
\end{proof}

\section{Strong maximality of affine equiangular sets in $L$}	\label{sec:strongmaximality}
In this section, we prove that every maximum affine equiangular set in $L$ with respect to the switching root $\sr$ is not strongly maximal.
Consequently, although all previously known sets of $57$ equiangular lines~\cite{Greaves2023,LMTY2025} in $\mathbb{R}^{18}$ were strongly maximal, there also exist such sets that are not strongly maximal.

\begin{theorem}	\label{thm:stronglymaximality}
	Every maximum affine equiangular set in $L$ with respect to the switching root $\sr$ is not strongly maximal.
\end{theorem}
\begin{proof}
	We fix a maximum affine equiangular set $\omega$ in $L$ with respect to the switching root $\sr$.
	By Theorem~\ref{thm:main}, we may assume that
	$\omega = \omega_0 \cup \omega_1 \cup \omega_2 \cup \omega_3$
	for some Latin square $\latin$ of order $6$ and three cliques $C_1, C_2, C_3$ of size $7$ in $K$, where $\omega_i$ ($i=0,1,2,3$) are defined in~\eqref{thm:main:omega0} and \eqref{thm:main:omegai}.

	It suffices to prove that there exists a vector $y \in \R L$ such that $(y,y) \leq 3$, $(y, \rho) =1 $ and $(y,x) \in \{0,1\}$ for every $x \in \omega$.
	For $I \in \binom{[6]}{3}$ and $\varepsilon \in \{ \pm 1\}$, we define
	\begin{align}\label{thm:stronglymaximality:y}
		y_{I,\varepsilon} := \left( \alpha_I, 0,0, \varepsilon \alpha_1,0 ,0 , \alpha_1 \right) \in \R L.
	\end{align}
	Then,
	$(y_{I,\varepsilon}, y_{I,\varepsilon}) = \frac{5}{2}$
	and
	$(y_{I,\varepsilon}, \rho) = 1$.
	Also, for every $x_{i,j,S_{ij}} \in \omega_0$,
	\begin{align*}
		(y_{I,\varepsilon}, x_{i,j,S_{ij}})
		 & = ( \alpha_I, \alpha_5^{(i)} ) + 0 + 0 + 0 + 0 + 0 + \frac{1}{2}
		= \begin{cases}
			  1 & \text{if } i \in I, \\
			  0 & \text{otherwise}.
		  \end{cases}
	\end{align*}
	For $(I',\varepsilon') \in C_i$ ($i=2,3$), we have $(y_{I,\varepsilon}, x^{(i)}_{I',\varepsilon',1}) \in \{0,1\}$.
	For $(I',\varepsilon') \in C_1$,
	\begin{align*}
		(y_{I,\varepsilon}, x^{(1)}_{I',\varepsilon',1})
		 & = ( \alpha_I, \alpha_{I'}) + 0 + 0 + 0 + 0 + 0 + \frac{1}{2}
		= |I \cap I'| - 1.
	\end{align*}
	This implies that
	\begin{align}
		(y_{I,\varepsilon}, x^{(1)}_{I',\varepsilon',1}) \in \{0,1\} \iff I \not\in \{ I', [6] \setminus I' \}.
	\end{align}
	Since $|\binom{[6]}{3}| - 2|C_1| = 20 - 14 = 6$, there exists $I \in \binom{[6]}{3}$ such that $I \not\in \{ I' , [6] \setminus I' \}$ for every $(I',\varepsilon') \in C_1$.
	Then, the vector $y := y_{I,1}$ satisfies the desired conditions.
	Therefore, $\omega$ is not strongly maximal.
\end{proof}

\begin{remark}
	As in the proof of Theorem~\ref{thm:stronglymaximality}, there exist many feasible vectors $y$, that is, vectors $y \in \R L$ such that $(y,y) \leq 3$, $(y, \rho) =1 $ and $(y,x) \in \{0,1\}$ for every $x \in \omega$.
	Indeed, we have $6 \cdot 2 = 12$ feasible vectors $y_{I,\varepsilon}$ defined in~\eqref{thm:stronglymaximality:y} for $I \in \binom{[6]}{3} \setminus \{ I', [6] \setminus I' : (I',\varepsilon') \in C_1 \}$ and $\varepsilon \in \{ \pm 1\}$.
	Similarly, by symmetry, we also have $12$ feasible vectors corresponding to $C_i$ for each $i = 2,3$.

	Using a computer, we investigated randomly chosen maximum affine equiangular sets $\omega$ in $L$ with respect to the switching root $\sr$.
	In the case where $\langle \omega \cup \{\sr\} \rangle = L$, our experiments indicate that the feasible vectors are exhausted by the $36$ vectors described above.
	In the case where $\langle \omega \cup \{\sr\} \rangle \neq L$,
	that is, $\langle \omega \cup \{\sr\} \rangle$ is isometric to $L_{54}$ by Proposition~\ref{prop:classifylattice},
	our experiments indicate that there are $90$ feasible vectors of norm $5/2$ and $108$ feasible vectors of norm $7/3$.
	One may ask whether adding such feasible vectors allows us to construct a larger affine equiangular set in dimension $20$.
	However, computational checks show that at most $12$ vectors can be added (for randomly chosen $\omega$).
	This yields a set of $57+12=69$ equiangular lines in $\mathbb{R}^{19}$,
	which does not improve the best known lower bound $N(19) \geq 72$.
\end{remark}

%\section*{Acknowledgements}
%Akihiro Munemasa was supported by JSPS KAKENHI Grant Number 25K07095.
%Ferenc Sz\"{o}ll\H{o}si was supported in part by JSPS KAKENHI Grant Number 24K06829.

\bibliographystyle{plain}
\bibliography{references}
\end{document}